\documentclass{amsart}
\usepackage[utf8]{inputenc}
\usepackage{amsmath}
\usepackage{amssymb}
\usepackage{amsthm}
\usepackage{enumitem}
\usepackage{todonotes}
\usepackage{hyperref}
\usepackage{url}
\usepackage{xcolor}
\usepackage{graphicx}
\usepackage{comment}
\usepackage{tikz}
\usetikzlibrary{cd}

\setlist[enumerate]{label=\rm{(\arabic*)}, ref=(\arabic*)}

\DeclareMathOperator{\Aut}{Aut}

\DeclareMathOperator{\Cay}{Cay}
\DeclareMathOperator{\DCay}{\overrightarrow{Cay}}

\DeclareMathOperator{\LSym}{Sym_L}
\DeclareMathOperator{\RSym}{Sym_R}
\DeclareMathOperator{\Comm}{Comm}

\newcommand{\N}{\mathbb{N}}
\newcommand{\Z}{\mathbb{Z}}

\newcommand*{\M}{\mathcal{M}}

\newcommand*{\eG}{\reflectbox{$G$}}

\newcommand*{\C}[2][X]{\mathcal{C}_{1}^{#1}(#2)}

\newtheorem{thm}{Theorem}[section]
\newtheorem{lemma}[thm]{Lemma}
\newtheorem{prop}[thm]{Proposition}
\newtheorem{cor}[thm]{Corollary}

\theoremstyle{definition}
\newtheorem{qu}[thm]{Question}
\newtheorem{defn}[thm]{Definition}

\title[On bireversible automata and commensurators of groups]{On bireversible automata and commensurators of groups in automorphisms of their Cayley graphs}
\author{Dominik Francoeur}

\begin{document}

\begin{abstract}
If $G$ is a finitely generated group and $X$ is a Cayley graph of $G$, denote by $\mathcal{C}_1^X(G)$ the subgroup of all automorphisms of $X$ commensurating $G$ and fixing the vertex corresponding to the identity.
Building on the work of Macedo\'{n}ska, Nekrashevych and Sushchansky, we observe that $\mathcal{C}_1^X(G)$ can be expressed as a directed union of groups generated by bireversible automata.
We use this to show that every cyclic subgroup of $\mathcal{C}_1^X(G)$ is undistorted and to obtain a necessary condition on $G$ for $\mathcal{C}_1^X(G)$ not to be locally finite.
As a consequence, we prove that several families of groups cannot be generated by bireversible automata and show that the set of groups generated by bireversible automata is strictly contained in the set of groups generated by invertible and reversible automata.
\end{abstract}

\maketitle

\section{Introduction}

\renewcommand{\thethm}{\Alph{thm}}

Let $G$ be a finitely generated group, let $X$ be a Cayley graph of $G$ with respect to some finite generating set, and let $\Aut(X)$ denote the group of automorphisms of $X$ seen as an unlabelled graph.
The group $\Aut(X)$ can be equipped with a natural topology that makes it a totally disconnected locally compact group, and through its action on itself by left multiplication, $G$ naturally embeds as a uniform lattice in $\Aut(X)$.

The group $\Aut(X)$ can depend significantly on the chosen generating set.
Indeed, by a theorem of Leemann and de la Salle \cite{delaSalleLeemann22a,delaSalleLeemann22}, for any finitely generated infinite group $G$, there always exists a finite generating set $S$ such that $G$ is of index at most $2$ in $\Aut(\Cay(G,S))$, while for some other generating sets the group of automorphisms could of course be much larger.
Nevertheless, the group $G$ can impose some restrictions on the properties of $\Aut(X)$ regardless of the choice of generating set.
For instance, it was shown by Trofimov \cite{Trofimov84} that if $G$ is a virtually nilpotent group, then $\Aut(X)$ is always compact-by-discrete.
It thus makes sense to study how the properties of a group $G$ can impact the possible properties of the group $\Aut(X)$.

Instead of the whole of $\Aut(X)$, one can consider the subgroup $\Comm_{\Aut(X)}(G)$ of elements of $\Aut(X)$ that commensurate the lattice $G$ (i.e. elements $\gamma\in \Aut(X)$ such that $\gamma G\gamma^{-1}\cap G$ is of finite index in both $G$ and $\gamma G \gamma^{-1}$, see Definition \ref{defn:Commensurator}).
In the case of a free group $F$ and a tree $T$, $F$ is a uniform tree lattice and $\Comm_{\Aut(T)}(F)$ is a dense subgroup of $\Aut(T)$ by a Theorem of Bass and Kulkarni \cite{BassKulkarni90}.
For general $G$ and $X$, this needs no longer hold, but $\Comm_{\Aut(X)}(G)$ is still an interesting subgroup of $\Aut(X)$ in its own right.

It follows from the fact that the action of $G$ on $X$ is free and transitive on the vertices of $X$ that $\Aut(X) = G\Aut_1(X)$, where $\Aut_1(X)$ denotes the elements of $\Aut(X)$ fixing the vertex corresponding to the identity in $G$.
Furthermore, $G\leq \Comm_{\Aut(X)}(G)$, so that $\Comm_{\Aut(X)}(G) = G \C{G}$, where $\C{G} = \Comm_{\Aut(X)}(G)\cap \Aut_1(X)$.
Thus, $\Comm_{\Aut(X)}(G)$ is discrete in $\Aut(X)$ if and only if $\C{G}$ is finite.

In the case of a free group $F$ and a tree $T$, it was shown by Macedo\'{n}ska, Nekrashevych and Sushchansky \cite{MacedonskaNekrashevychSushchansky00} that $\C[T]{F}$ corresponds to the set of all so-called \emph{bireversible} automorphisms, which are automorphisms whose action on the tree $T$ can be described by a bireversible Mealy automaton.
A Mealy automaton is a deterministic finite-state transducer, that is, a machine with a finite number of states reading as input a letter in some finite alphabet, returning as output a letter in the same alphabet and changing its current state depending on both the letter read and the currently active state (see Definition \ref{defn:MealyAutomaton} for a more precise definition).
There are two basic operations that one can perform on a Mealy automaton, namely taking the inverse and taking the dual, that will result in automata that are not necessarily Mealy automata.
If the inverse of a Mealy automaton is again a Mealy automaton, then it is said to be \emph{invertible}, and if the inverse of the dual of a Mealy automaton is again a Mealy automaton, then it is said to be \emph{reversible}.
If all possible combinations of iteratively taking inverses or duals of a Mealy automaton always yields a Mealy automaton, then that automaton is said to be \emph{bireversible}.
Note in particular that a bireversible Mealy automaton is both invertible and reversible, but that there are invertible and reversible Mealy automata that are not necessarily bireversible.

To every invertible Mealy automaton, one can associate a group (see Definition \ref{defn:AutomatonGroup}).
Currently, however, we only know very little about which groups can be generated by Mealy automata and how the properties of the automaton influence the properties of the group.
It turns out that even small automata can generate fairly complicated groups.
For instance, it was shown by Grigorchuk that using a Mealy automaton with only five states on an alphabet of two letters, one can generate an infinite finitely generated periodic group \cite{Grigorchuk80}.
Later, he also showed that the same group is in fact of intermediate word growth \cite{Grigorchuk83}, thus providing the first example of such a group.
Grigorchuk's striking results have brought forth a surge of interest in groups generated by Mealy automata that continues unabated to this day.

Since Grigorchuk's example was generated by an invertible but not reversible automaton, many of the subsequent works have focused on this class, and groups generated by bireversible Mealy automata (henceforth, we shall often drop the word Mealy and call them simply bireversible automata) have received comparatively less attention.
Nevertheless, in light of Macedo\'{n}ska, Nekrashevych and Sushchansky's result, it would be useful to have a better understanding of these groups.
It is known that finitely generated free groups \cite{SteinbergVorobetsVorobets11} and groups of the form $A\wr \Z$ for $A$ a non-trivial finite abelian group \cite{Francoeur23} can be generated by bireversible automata.
Thanks to the work of Glasner and Mozes \cite{GlasnerMozes05}, who discovered a connection between bireversible automata and non-positively curved square complexes, we also have examples of groups with Kazhdan's property (T) generated by bireversible automata.
On the other hand, we know from the work of Klimann \cite{Klimann18} that a virtually nilpotent group generated by a bireversible automaton must be finite, and by the work of the author and Mitrofanov \cite{FrancoeurMitrofanov21}, a group generated by such an automaton must either be periodic or contain a free subsemigroup on two generators.
Beyond this, however, we currently know very little.

The aim of the current article is to leverage the connection between groups generated by bireversible automata, which we call here bireversible groups, and commensurators in automorphisms of Cayley graphs to learn more about both.
To this end, we first note that Macedo\'{n}ska, Nekrashevych and Sushchansky's result can easily be adapted to all finitely generated groups $G$ and Cayley graphs $X$ (Theorem \ref{thm:GeneralisedMNS}), and we use this to show that any group $G$ for which $\C{G}$ is not locally finite must map onto an infinite bireversible group.
More precisely, we show the following theorem.

\begin{thm}[Theorem \ref{thm:GroupsWithBigCommensurators}]\label{thm:GroupsWithBigCommensuratorsIntro}
Let $G$ be a finitely generated group and let $X$ be a Cayley graph of $G$ such that
\[\C{G}= \Comm_{\Aut(X)}(G) \cap \Aut_1(X)\]
is not locally finite.
Let $D_G$ be the largest subgroup of $G$ that is normalised by $\Comm_{\Aut(X)}(G)$.
Then, $G/D_G$ is residually infinite bireversible (i.e. for every $g\in G$ with $g\notin D_G$, there exists $D_G\leq N\trianglelefteq G$ such that $g\notin N$ and $G/N$ is an infinite bireversible group).
\end{thm}

As a corollary, we find that if $G$ is such that every proper quotient is virtually nilpotent (in particular, if $G$ is a just infinite group), then $\C{G}$ has to be locally finite, unless $G$ is bireversible.

\begin{cor}[Corollary \ref{cor:JustInfiniteIsBirOrSmall}]\label{cor:JustInfiniteIsBirOrSmallIntro}
If $G$ is an infinite finitely generated group such that every proper quotient is virtually nilpotent, then either $G$ is a bireversible group or 
\[\C{G}= \Comm_{\Aut(X)}(G) \cap \Aut_1(X)\]
is locally finite for all Cayley graphs $X$ of $G$.
\end{cor}

From the same considerations about commensurators of bireversible groups, we also observe that an infinite bireversible group $G$ always acts geometrically on a graph $X$ such that the commensurator of $G$ in $\Aut(X)$ has a finitely generated subgroup that is not discrete, which implies that $\Aut(X)$ is not compact-by-discrete.
This simple observation, when combined with existing results in the literature, allows us to prove that several groups cannot be bireversible.

\begin{prop}[Corollary \ref{cor:ListOfGroups}]\label{prop:ListOfGroupsIntro}
The following groups are not bireversible:
\begin{enumerate}
\item irreducible lattices in center-free, real semi-simple Lie groups without compact factors and not locally isomorphic to $\mathrm{SL}_2(\mathbb{R})$,
\item uniform lattices in $\mathrm{PSL}_2(\mathbb{R})$,
\item the group $\mathrm{Out}(F_n)$ of outer automorphisms of a free group of rank $n\geq 3$,
\item topologically rigid hyperbolic groups,
\item hyperbolic groups whose visual boundary is homeomorphic to an $n$-sphere with $n\leq 3$ or to a Sierpinski carpet,
\item fundamental groups of closed irreducible oriented 3-manifolds with non-trivial geometric decomposition.
\end{enumerate}
\end{prop}

Lastly, we prove that cyclic subgroups of bireversible groups are undistorted.

\begin{thm}[Theorem \ref{thm:UndistortedCyclicSubgroups}]\label{thm:UndistortedCyclicSubgroupsIntro}
Every cyclic subgroup of a bireversible group is undistorted.
\end{thm}

We derive two interesting consequences from this.
The first is that the set of bireversible groups is strictly smaller than the set of groups generated by invertible and reversible automata.

\begin{cor}[Corollary \ref{cor:InvRevNotBir}]\label{cor:InvRevNotBirIntro}
Let $\mathbf{Bir}$ be the set of groups generated by bireversible automata and let $\mathbf{InvRev}$ be the set of groups generated by invertible and reversible automata.
Then, $\mathbf{Bir}\subsetneq \mathbf{InvRev}$.
\end{cor}

The second is that for any group $G$ and any Cayley graph $X$, cyclic subgroups of the commensurator $\C{G}$ are undistorted.

\begin{cor}[Corollary \ref{cor:UndistortedCyclicInCommensurator}]\label{cor:UndistortedCyclicInCommensuratorIntro}
Let $G$ be a finitely generated group and let $X$ be a Cayley graph of $G$.
Then, every cyclic subgroup of
\[\C{G} = \Comm_{\Aut(X)}(G)\cap \Aut_1(X)\]
is undistorted.
\end{cor}

We would like to mention that we later learned that Pierre-Emmanuel Caprace and Adrien Le Boudec had already obtained this last result (Corollary \ref{cor:UndistortedCyclicInCommensuratorIntro}) by completely different methods, although it was never published.

\subsection*{Organisation of the article}

In Section \ref{sec:Preliminaries}, we collect the background definitions and results about Mealy automata and automaton groups that we will require for the rest of the article.
In Section \ref{sec:Commensurator}, we generalise the result of Macedo\'{n}ska, Nekrashevych and Sushchansky and prove Theorem \ref{thm:GroupsWithBigCommensuratorsIntro} and Corollary \ref{cor:JustInfiniteIsBirOrSmallIntro}.
In Section \ref{sec:GraphicalDiscreteness}, we observe that bireversible groups act geometrically on a graph whose automorphism group is not compact-by-discrete and prove Proposition \ref{prop:ListOfGroupsIntro}.
Lastly, in Section \ref{sec:Distortion}, we study the distortion of cyclic subgroups in bireversible groups and prove Theorem \ref{thm:UndistortedCyclicSubgroupsIntro}, Corollary \ref{cor:InvRevNotBirIntro} and Corollary \ref{cor:UndistortedCyclicInCommensuratorIntro}.

\subsection*{Acknowledgements}

The author would like to thank Adrien Le Boudec and Mikael de la Salle for useful discussions and comments.

\section{Preliminaries}\label{sec:Preliminaries}

\renewcommand{\thethm}{\arabic{section}.\arabic{thm}}

\subsection{Mealy automata and automaton groups}

We begin by the definition of a Mealy automaton.

\begin{defn}\label{defn:MealyAutomaton}
A Mealy automaton is a tuple $\M=(A,Q, \lambda, \rho)$, where $A$ and $Q$ are finite sets, called respectively the \emph{alphabet} and the \emph{state set}, and $\lambda\colon Q\times A \rightarrow A$, $\rho \colon Q\times A \rightarrow Q$ are maps called respectively the \emph{output map} and the \emph{transition map}.
We say that $\M$ is
\begin{enumerate}[label=(\roman*)]
\item \emph{invertible} if for all $q\in Q$, the map $\lambda_q\colon A \rightarrow A$, $a\mapsto \lambda(q,a)$ is a bijection,
\item \emph{reversible} if for all $a\in A$, the map $\rho_a\colon Q \rightarrow Q$, $q\mapsto \rho(q,a)$ is a bijection,
\item \emph{bireversible} if $\M$ is invertible, reversible, and if the map $\delta\colon Q\times A \rightarrow A\times Q$, $(q,a) \mapsto (\lambda(q,a), \rho(q,a))$ is a bijection.
\end{enumerate}
\end{defn}

Let $\M=(A,Q,\lambda,\rho)$ be a Mealy automaton, and let us denote by $Q^*$ and $A^*$ the free monoids on $Q$ and $A$, respectively.
For every $q\in Q$, the map $\lambda_q\colon A \rightarrow A$ can be extended to a map $\lambda_q\colon A^* \rightarrow A^*$ through the recursive formula $\lambda_q(aw) = \lambda_q(a)\lambda_{\rho(q,a)}(w)$ for all $a\in A$ and $w\in A^*$ (note that we use the same symbol both for the original map and its extension, but this should not cause confusion, since one is an extension of the other).
Similarly, for every $a\in A$, one can extend the map $\rho_a\colon Q \rightarrow Q$ to a map $\rho_q\colon Q^* \rightarrow Q^*$ by $\rho_a(vq) = \rho_{\lambda(q,a)}(v)\rho_a(q)$ for all $q\in Q$ and $v\in Q^*$.
Note that this map is more natural when the word $vq$ is read from right to left.

If $\M$ is invertible, then it is easy to see that the maps $\lambda_q\colon A^* \rightarrow A^*$ are bijections for all $q\in Q$.
Therefore, we have a natural homomorphism $f_Q\colon F_Q \rightarrow \LSym(A^*)$,
where $F_Q$ denotes the free group on $Q$ and $\LSym(A^*)$ denotes the group of bijections of $A^*$ with composition taken from right to left, so that it acts on the left on $A^*$.
Similarly, if $\M$ is reversible, we can define a homomorphism $f_A\colon F_A \rightarrow \RSym(Q^*)$, where $\RSym(Q^*)$ is the group of bijections of $Q^*$ with composition taken from left to right, so that the natural action is a right action.
This allows us to define the group and dual group generated by an automaton.

\begin{defn}\label{defn:AutomatonGroup}
Let $\M=(A,Q,\lambda, \rho)$ be a Mealy automaton.
If $\M$ is invertible, the group generated by $\M$ is the group $G_\M = F_Q/\ker(f_Q)$.
If $\M$ is invertible, the dual group generated by $\M$ is the group $\eG_\M = F_A/\ker(f_A)$.
\end{defn}

\subsection{Bireversible automaton groups}

For bireversible automata, there exists a different characterisation of both the group and the dual group.
To state it, we first define a group that we will call here the \emph{fundamental group} of a bireversible automaton.
We call it so because it is in fact the fundamental group of a directed VH-T-complex associated to the automaton, first defined by Glasner and Mozes in \cite{GlasnerMozes05}.

\begin{defn}
Let $\M=(A,Q,\lambda, \rho)$ be a bireversible automaton.
Its fundamental group is the group
\[\pi_1(\M) = \langle Q,A \mid qa = \lambda(q,a)\rho(q,a) \quad \forall q\in Q, a\in A \rangle.\]
\end{defn}

The following theorem, which is a combination of Theorems 3.5 and 3.6 in \cite{BondarenkoKivva22}, clarifies the structure of the fundamental group of a bireversible automaton and explains how to extract the group and dual group generated by the automaton from it.

\begin{thm}[cf. \protect{\cite[Theorems 3.5 and 3.6]{BondarenkoKivva22}}]\label{thm:AlternativeDefBireversibleGroup}
Let $\M=(A,Q,\lambda, \rho)$ be a bireversible automaton and let $\pi_1(\M)$ be its fundamental group.
Then,
\begin{enumerate}
\item the subgroups $\langle Q \rangle, \langle S \rangle \leq \pi_1(\M)$ are free, with free basis $Q$ and $S$, respectively,\label{item:TwoFreeGroups}
\item for every $p\in \pi_1(\M)$, there exist unique elements $g,h\in \langle Q \rangle$ and $v,w\in \langle A \rangle$ such that $gv = p = wh$,\label{item:UniqueProduct}
\item if $N_Q \leq F_Q = \langle Q \rangle$ and $N_A \leq F_A = \langle A \rangle$ are the largest normal subgroups of $\pi_1(\M)$ contained in $F_Q$ and $F_A$, respectively, then $G_\M = F_Q/N_Q$ and $\eG_\M = F_A/N_A$.\label{item:MaxNormal}
\end{enumerate}
\end{thm}
The alternative way of defining the groups generated by a bireversible automaton given by Theorem \ref{thm:AlternativeDefBireversibleGroup} will often be more convenient for us than the original definition.

We are interested in understanding which groups can be generated by bireversible automata.
We will call such groups \emph{bireversible groups}.

\begin{defn}
A group $G$ will be called bireversible if there exists a bireversible automaton $\M$ such that $G$ is isomorphic to $G_\M$.
\end{defn}

Note that the automaton $\M$ in the previous definition is not necessarily unique.
Note also that one could conceivably define a dual bireversible group as a group isomorphic to $\eG_\M$ for some bireversible automaton $\M$.
However, as the next proposition shows, this is not needed, since every dual group generated by a bireversible automaton is a bireversible group and vice-versa.

\begin{prop}\label{prop:DualIsBireversible}
Let $\M=(A,Q,\lambda, \rho)$ be a bireversible automaton.
There exists a bireversible automaton $\overline{\M} = (Q,A,\overline{\rho}, \overline{\lambda})$ such that $G_\M = \eG_{\overline{\M}}$ and $\eG_\M = G_{\overline{\M}}$.
\end{prop}
\begin{proof}
Since $\M$ is bireversible, for every $q\in Q$ and $a\in A$, there exist unique elements $q'\in Q$ and $a'\in A$ such that $\lambda(q',a') = a$ and $\rho(q',a') = q$.
We define the maps $\overline{\lambda}\colon A\times Q \rightarrow A$ and $\overline{\rho}\colon A\times Q \rightarrow Q$ by $\overline{\lambda}(a,q) = a'$ and $\overline{\rho}(a,q) = q'$.
It is almost immediate that the automaton $\overline{\M} = (A,Q,\overline{\rho}, \overline{\lambda})$ is bireversible.
Notice that
\[\pi_1(\overline{\M}) = \langle A,Q \mid aq = q'a' \quad \forall q\in Q, a\in A \rangle = \pi_1(\M).\]
It follows that $G_{\overline{\M}} = F_A/N_A = \eG_\M$ and $\eG_{\overline{\M}} = F_Q/N_Q = G_\M$.
\end{proof}

\subsection{Bireversible automorphisms of trees and commensurators of free groups}\label{subsec:CommensuratorFree}

In this subsection, we will explain the link between bireversible automata and the commensurator of a free group inside the group of automorphisms of its oriented Cayley graph, which was first established by Macedo\'{n}ska, Nekrashevych and Sushchansky \cite{MacedonskaNekrashevychSushchansky00}.
In fact, it is in this context that bireversible automata were first defined and studied.

Before we do so, however, let us first recall the definition of an oriented Cayley graph.

\begin{defn}
Let $G$ be a finitely generated group and let $S\subseteq G$ be a finite set of generators (not necessarily symmetric).
The (unlabelled) oriented Cayley graph of $G$ with respect to $S$ is the graph $\DCay(G,S)$ whose vertex set is $G$ and whose edge set is $E=\{(g,gs) \mid g\in G, s\in S\}$.
\end{defn}

In other words, the oriented Cayley graph of $G$ is the Cayley graph of $G$ where all the edges corresponding to elements of $S^{-1}$ have been removed.
Note that this graph is still connected, but it is generally not strongly connected.

Now, let $\M=(A,Q,\lambda,\rho)$ be a bireversible automaton.
By Theorem \ref{thm:AlternativeDefBireversibleGroup} \ref{item:UniqueProduct}, for any $g\in \pi_1(\M)$ and $v\in F_A \leq \pi_1(\M)$, there exist $h\in F_Q$ and $w\in F_A$ such that $gv= wh$.
The uniqueness of this decomposition implies that the map $\cdot\colon \pi_1(\M)\times F_A \rightarrow F_A$ given by $(g,v) \mapsto w$ is a left group action.
Furthermore, this action is an action by isometries with respect to the word metric $d_A\colon F_A\times F_A \rightarrow \N$ on $F_A$ induced by the generating set $A$.
Indeed, if $v=a_1\dots a_n$, where $a_i\in A^{\pm 1}$ and the word $a_1\dots a_n$ is reduced, and if $g\in \pi_1(\M)$ is an arbitrary element, then it follows from the definition of $\pi_1(\M)$ and Theorem \ref{thm:AlternativeDefBireversibleGroup} that there exist $a'_i\in A^{\pm 1}$, $w\in F_A$ and $h\in F_Q$ such that $g=wh$ and  $g\cdot v = wa'_1 \dots a'_n$.
Therefore, $d_A(g\cdot 1,g\cdot v) = d_A(w,wa'_1\dots a'_n)\leq n=d_A(1,v)$, and since this is valid for any $g\in \pi_1(\M)$ and $v\in F_A$, the other inequality is obtained by applying $h^{-1}$ to $a'_1\dots a'_n$.
Furthermore, notice that for all $g\in \pi_1(\M)$, $a_1, \dots, a_n, a'_1, \dots a'_n \in A^{\pm 1}$ and $w\in F_A$ with $g\cdot (a_1\dots a_n) = wa'_1 \dots a'_n$, we have $a'_i\in A$ if and only if $a_i \in A$, so that the orientation induced by $A$ is preserved.
It follows that every element of $\pi_1(\M)$ defines an automorphism of the oriented Cayley graph $\DCay(F_A,A)$, which is an oriented $2|A|$-regular rooted tree where every vertex has $|A|$ outgoing edges and $|A|$ incoming edges.
Therefore, if we denote $\overrightarrow{T}_A = \DCay(F_A,A)$ the oriented Cayley graph of $F_A$, the action $\cdot\colon \pi_1(\M)\times F_A \rightarrow F_A$ induces a homomorphism $\Psi_\M \colon \pi_1(\M) \rightarrow \Aut(\overrightarrow{T}_A)$.


An element $g\in \pi_1(\M)$ is in the kernel of $\Psi_\M$ if and only if $g\in F_Q$ and for all $v\in F_A$ there exists $h\in F_Q$ such that $gv=vh$.
In particular, for all $v, v'\in F_A$, we have $gvv'= vhv' = vv'h'$ for some $h'\in F_Q$, so that for all $v'\in F_A$, $hv'=v'h'$.
This means that $h$ is also in the kernel of $\Psi_\M$, from which we deduce that the kernel is the largest normal subgroup of $\pi_1(\M)$ that is contained in $F_Q$.
Thus, the map $\Psi_\M$ descends to an injective homomorphism $\psi_\M \colon \pi_1(\M)/N_Q \rightarrow \Aut(\overrightarrow{T}_A)$, where $N_Q$ is the largest normal subgroup of $\pi_1(\M)$ contained in $F_Q$.
By Theorem \ref{thm:AlternativeDefBireversibleGroup} \ref{item:MaxNormal}, $\pi_1(\M)/N_Q$ contains $G_\M$ as a subgroup, and it is easy to see from the definitions that $\psi_\M(G_\M)\leq \Aut_1(\overrightarrow{T}_A)$, where $\Aut_1(\overrightarrow{T}_A)$ is the stabiliser of the identity in the group $\Aut(\overrightarrow{T}_A)$ of automorphisms of $\overrightarrow{T}_A$.

Elements of $\Aut_1(\overrightarrow{T}_A)$ which are in the image of a map $\psi_\M$ for some bireversible automaton $\M$ are called bireversible automorphisms.

\begin{defn}
Let $A$ be a finite set and let $\overline{T}_A$ denote the oriented Cayley graph of the free group $F_A$ with respect to $A$.
An automorphism $f\in \Aut_1(\overrightarrow{T}_A)$ is said to be bireversible if there exists a bireversible automaton $\M$ such that $f\in \psi_\M(G_\M)$.
\end{defn}

Note that through $\psi_\M$, $F_A$ can be seen as a subgroup of $\Aut(\overrightarrow{T}_A)$, which simply corresponds to the action of $F_A$ on itself by left multiplication.
The result of Macedo\'{n}ska, Nekrashevych and Sushchansky is that the elements of $\Aut(\overline{T}_A)$ that both commensurate $F_A$ and fix the identity are exactly the bireversible automorphisms.
Before we properly state their theorem, let us recall what it means to commensurate a subgroup.

\begin{defn}\label{defn:Commensurator}
Let $G$ be a group and let $H\leq G$ be a subgroup.
The commensurator of $H$ in $G$ is the subgroup
\[\Comm_G(H) = \left\{g\in G \mid [H:H\cap gHg^{-1}]<\infty, [gHg^{-1}: H \cap gHg^{-1}]<\infty\right\}.\]
\end{defn}

We can now state Macedo\'{n}ska, Nekrashevych and Sushchansky's result.

\begin{thm}[\protect{\cite[Theorem 6]{MacedonskaNekrashevychSushchansky00}}]\label{thm:MNS}
Let $A$ be a finite set, let $\overrightarrow{T}_A$ be the oriented Cayley graph of the free group $F_A$ with respect to $A$ and let
\[\C[\overrightarrow{T}_A]{F_A}= \Comm_{\Aut(\overrightarrow{T}_A)}(F_A) \cap \Aut_1(\overrightarrow{T}_A)\]
be the subgroup of elements of $\Aut_1(\overrightarrow{T}_A)$ commensurating $F_A$, where $F_A$ is naturally seen as a subgroup of $\Aut(\overrightarrow{T}_A)$ through its action by left multiplication.
Let $B$ be the set of all bireversible tree automorphisms of $\overrightarrow{T}_A$.
Then, $B=\C[\overrightarrow{T}_A]{F_A}$.
\end{thm}

\subsection{Automorphisms of quotient graphs}
In this subsection, we record the fact that if a graph is a quotient of the oriented Cayley graph $\overrightarrow{T}_A$ of the free group $F_A$ by a normal subgroup of $F_A$, then its group of automorphisms is a quotient of a subgroup of the group of automorphisms of $\overrightarrow{T}_A$, something that will be very useful to us later on.

In order to do this, we begin by recalling the definition of a marked group.

\begin{defn}
A \emph{marked group} is a triple $(G,A,\varphi)$, where $G$ is a group, $A$ is a set and $\varphi\colon F_A\rightarrow G$ is an epimorphism called a \emph{marking}.
A \emph{finitely generated marked group} is a marked group $(G,A,\varphi)$ such that $A$ is finite.
\end{defn}

If $(G,A,\varphi)$ is a finitely generated marked group, the kernel of $\varphi$, being a subgroup of the free group $F_A$, acts on the oriented Cayley graph $\overrightarrow{T}_A$ of $F_A$ by left multiplication.
Therefore, one can consider the graph $\ker(\varphi) \backslash \overrightarrow{T}_A$.
Note that the difference between this graph and the oriented Cayley graph of $G$ with respect to the generating set $\varphi(A)$ is that $\ker(\varphi) \backslash \overrightarrow{T}_A$ has multiple edges if $\varphi$ is not injective on $A$.
%

It is immediate from the definition that automorphisms of $\overrightarrow{T}_A$ normalising $\ker(\varphi)$ descend to automorphisms of $\ker(\varphi) \backslash \overrightarrow{T}_A$.
Furthermore, by covering space theory, one can see that two such automorphisms descend to the same one if and only if they differ by an element of $\ker(\varphi)$.
We record these facts in the following proposition.

\begin{prop}\label{prop:NormaliserPassesToQuotient}
Let $(G,A,\varphi)$ be a finitely generated marked group, let $\overrightarrow{T}_A$ be the directed Cayley graph of $F_A$ with respect to $A$ and let $X = \ker(\varphi)\backslash \overrightarrow{T}_A$.
Let $p_\varphi\colon \overrightarrow{T}_A \rightarrow X$ be the canonical quotient morphism.
There exists a unique group homomorphism $\zeta_\varphi\colon N_{\Aut(\overrightarrow{T}_A)}(\ker(\varphi)) \rightarrow \Aut(X)$ such that the the following diagram commutes:
\begin{center}
\begin{tikzcd}
\overrightarrow{T}_A \arrow[d, "p_\varphi"] \arrow[r, "g"] & \overrightarrow{T}_A \arrow[d, "p_\varphi"]\\
X \arrow[r, "\zeta_\varphi(g)"] & X
\end{tikzcd}
\end{center}
where $N_{\Aut(\overrightarrow{T}_A)}(\ker(\varphi))$ denotes the normaliser of $\ker(\varphi)$ in $\Aut(\overrightarrow{T}_A)$.
Furthermore, this homomorphism is surjective and its kernel is $\ker(\varphi)$.
In particular, $G\leq \Aut(X)$, and $\zeta_\varphi$ restricts to an isomorphism between $N_{\Aut(\overrightarrow{T}_A)}(\ker(\varphi))\cap \Aut_1(\overrightarrow{T}_A)$ and $\Aut_1(X)$.
\end{prop}
\begin{proof}
The existence and uniqueness of the homomorphism $\zeta_\varphi$ are immediate from the definition.

Since $F_A$ acts freely on $\overrightarrow{T}_A$, $p_\varphi\colon \overrightarrow{T}_A\rightarrow X$ is a covering, whose group of deck transformations is $\ker(\varphi)$.
Therefore, every automorphism of $X$ lifts to an automorphism of $\overrightarrow{T}_A$, which must necessarily lie in $N_{\Aut(\overrightarrow{T}_A)}(\ker(\varphi))$, and this lift is unique up to translation by an element of the group of deck transformations $\ker(\varphi)$.
It follows that the image of $\varphi$ is $\Aut(X)$ and that its kernel is $\ker(\varphi)$.
\end{proof}

\subsection{Useful constructions with bireversible automata}

In this last subsection of our preliminaries, we present three constructions to produce new bireversible automata from old ones that will be useful to us later on.
The first one is the disjoint union of bireversible automata over the same alphabet.

\begin{prop}\label{prop:UnionOfAutomata}
Let $\M_i=(A,Q_i,\lambda_i,\rho_i)$ for $i\in \{1,\dots, n\}$ be bireversible automata over the same alphabet $A$.
Let $Q=\bigsqcup_{i=1}^{n} Q_i$, and let us define $\lambda\colon Q\times A\rightarrow A$ and $\rho\colon Q\times A \rightarrow Q$ by $\lambda(q_i,a) = \lambda_i(q_i,a)$ and $\rho(q_i,a) = \rho_i(q_i,a)$ for all $a\in A$, $i\in \{1,\dots, n\}$ and $q_i\in Q_i$.
Then, the automaton $\M:=\bigsqcup_{i=1}^{n}\M_i=(A, Q, \lambda, \rho)$ is bireversible, and if we denote by $p_Q\colon F_Q \rightarrow G_\M$ the canonical quotient map, then $p_Q(F_{Q_i})= G_{\M_i}$.
Furthermore, if a subgroup $N\trianglelefteq F_A$ is normal in $\pi_1(\M_i)$ for all $i\in \{1,\dots, n\}$, then $N$ is normal in $\pi_1(\M)$.
\end{prop}
\begin{proof}
A routine check allows one to prove that the automaton $\M$ is indeed bireversible.
From the presentation, we see that its fundamental group $\pi_1(\M)$ is the amalgamated free product of the fundamental groups $\pi_1(\M_i)$ over the subgroup $F_A$.
It follows immediately that a subgroup $N\trianglelefteq F_A$ that is normal in each $\pi_1(\M_i)$ will be normal in $\pi_1(\M)$.

To prove the claim that $p_Q(F_{Q_i}) = G_{\M_i}$, let us denote by $N_Q\trianglelefteq F_Q$ the kernel of $p_Q$ and by $N_{Q_i}\trianglelefteq F_{Q_i}$ the subgroups such that $G_{\M_i} = F_{Q_i}/N_{Q_i}$.
By Theorem \ref{thm:AlternativeDefBireversibleGroup} \ref{item:MaxNormal}, $N_Q$ is the largest normal subgroup of $\pi_1(\M)$ contained in $F_Q$.
Therefore, $N_Q\cap \pi_1(\M_i)$ is a normal subgroup of $\pi_1(\M_i)$ contained in $F_Q\cap \pi_1(\M_i) = F_{Q_i}$, from which we deduce that $N_Q\cap \pi_1(\M_i) \leq N_{Q_i}$.
On the other hand, notice that if $g\in N_{Q_i}$, $h\in F_Q$ and $v\in F_A$, then by Theorem \ref{thm:AlternativeDefBireversibleGroup} \ref{item:UniqueProduct},
\[v^{-1}h^{-1}ghv = v^{-1}h^{-1}gv'h' = v^{-1}h^{-1}v'g'h' = v^{-1}vh'^{-1}g'h' = h'^{-1}g'h'\]
for some $h'\in F_Q$, $g'\in N_{Q_i}$ and $v'\in F_A$.
In other words, the set of conjugates of elements of $N_{Q_i}$ by elements of $F_Q$ is invariant under conjugation by $F_A$.
It follows that the smallest normal subgroup of $\pi_1(\M)$ containing $N_{Q_i}$ is contained in $F_Q$, and therefore that $N_{Q_i}\leq N_Q$.
This finishes showing that $N_Q \cap \pi_1(\M) = N_Q\cap F_{Q_i} = N_{Q_i}$, from which we deduce that $p_Q(F_{Q_i}) = G_{\M_i}$.
\end{proof}

The second construction that we wish to introduce is one related to taking subgroups in the  group generated by the automaton.

\begin{prop}\label{prop:AutomatonCorrespondingToSubgroup}
Let $\M=(A,Q,\lambda,\rho)$ be a bireversible automaton and let $H\leq G_\M$ be a finitely generated subgroup.
Then, there exists a bireversible automaton $\M' = (A, S,\lambda',\rho')$ such that $\psi_{\M'}(\pi_1(\M')/N_S) = \psi_{\M}(\langle H, F_A \rangle)$, where the maps $\psi_{\M}\colon \pi_1(\M)/N_Q \rightarrow \Aut(\overrightarrow{T}_A)$ and $\psi_{\M'}\colon \pi_1(\M')/N_S \rightarrow \Aut(\overrightarrow{T}_A)$ are the ones described in Section \ref{subsec:CommensuratorFree} and the subgroups $N_Q$ and $N_S$ are the ones of Theorem \ref{thm:AlternativeDefBireversibleGroup} \ref{item:MaxNormal}.
\end{prop}
\begin{proof}
Let $p_Q \colon F_Q \rightarrow G_\M$ be the canonical quotient map.
Since $H$ is finitely generated, there exists a finite subset $T\subseteq F_Q$ such that $\langle p_Q(T) \rangle = H$.
Let us define
\[S=\bigcup_{v\in F_A} \{s\in F_Q \mid \exists \, t\in T, v,w\in F_A \text{ such that } tv=ws \}.\]
It follows from this definition that any element of $S$ must be of the same word length as some element of $T$.
Therefore, since $T$ is finite, $S$ must also be finite.

Let us define $\lambda'\colon S\times A \rightarrow A$ and $\rho'\colon S\times A \rightarrow S$ by $\lambda'(s,a) = b$ and $\rho'(s,a)=t$, where $b\in A$ and $t\in S$ are the unique elements such that $sa=bt$ (note that $t\in S$ from the definition of $S$).
It follows from the uniqueness of this decomposition that $\M'=(A,S,\lambda',\rho')$ is a bireversible automaton.

It is immediate from the definition that we have a homomorphism $f\colon \pi_1(\M')\rightarrow \pi_1(\M)$ such that $f(v)=v$ for all $v\in F_A$ and $f(s)=s$ for all $s\in S$.
Note that if $g\in N_S$, then for all $v\in F_A$, there exists $g'\in N_S$ such that $gv=vg'$, and thus $f(g)v=vf(g')$.
It follows that $f(N_S)\leq N_Q$, from which we conclude that the homomorphism $f$ induces a homomorphism $\overline{f}\colon \pi_1(\M')/N_S \rightarrow \pi_1(\M)/N_Q$ satisfying $\overline{f}(v) = v$ for all $v\in F_A$ and $\overline{f}(p_{S}(s)) = p_Q(s)$ for all $s\in S$, where $p_{S}\colon F_S \rightarrow G_{\M'}$ is the canonical quotient map.
From this, we see that $\psi_{\M'}(\pi_1(\M')) = \psi_\M(\overline{f}(\pi_1(\M')))$ and it suffices to show that $\overline{f}(\pi_1(\M')) = \langle H, F_A \rangle$.

From the fact that $\overline{f}$ acts as the identity on $F_A$ and sends $G_{\M'}$ to $\langle p_Q(S) \rangle$, we already see that $\langle H, F_A \rangle \leq \overline{f}(\pi_1(\M'))$.
To see the opposite containment, it suffices to show that $\langle H, F_A \rangle \cap G_\M$ must contain $p_Q(S)$.
However, this is immediate from the definition of $S$ and the fact that $\langle H, F_A \rangle$ must contain $p_Q(T)$ and $F_A$.
\end{proof}

The third condition that we are going to need is one that allows us to change the natural  generating sets of the automaton group and the dual automaton group to symmetric ones.

\begin{prop}\label{prop:MonoidGenerates}
Let $\M=(A,Q,\lambda, \rho)$ be a bireversible automaton.
There exists a bireversible automaton $\hat{\M}=(\hat{A},\hat{Q},\hat{\lambda},\hat{\rho})$ such that $p_{\hat{Q}}(\hat{Q}^*)=G_{\hat{\M}} \cong G_\M$ and $p_{\hat{A}}(\hat{A}^*) = \eG_{\hat{\M}} \cong \eG_\M$, where $p_{\hat{Q}}\colon F_{\hat{Q}}\rightarrow G_{\hat{\M}}$ and $p_{\hat{A}}\colon F_{\hat{A}}\rightarrow \eG_{\hat{\M}}$ are the canonical quotient maps.
\end{prop}
\begin{proof}
Let us set $\hat{Q} = Q\sqcup \overline{Q}$ and $\hat{A} = A \sqcup \overline{A}$.
Let $\iota_Q \colon \hat{Q}\rightarrow F_Q$ be defined by $\iota_Q(q)=q$ and $\iota_Q(\overline{q}) = q^{-1}$ for all $q\in Q$, and let $\iota_A\colon \hat{A}\rightarrow F_A$ be defined analogously.
We define the maps $\hat{\lambda}\colon \hat{Q}\times \hat{A} \rightarrow \hat{A}$ and $\hat{\rho}\colon \hat{Q}\times \hat{A} \rightarrow \hat{Q}$ as the unique maps such that $\iota_Q(\hat{q})\iota_A(\hat{a}) = \iota_A(\hat{\lambda}(\hat{q},\hat{a}))\iota_Q(\hat{\rho}(\hat{q},\hat{a}))$ in $\pi_1(\M)$ for all $\hat{q}\in \hat{Q}$ and $\hat{a}\in \hat{A}$.
It is not hard to see that these maps are well-defined, and the uniqueness of the decomposition in $\pi_1(\M)$ implies that the automaton $\hat{\M} = (\hat{A}, \hat{Q}, \hat{\lambda}, \hat{\rho})$ is bireversible.

Let $\kappa\colon \pi_1(\hat{\M})\rightarrow \pi_1(\M)$ be the homomorphism induced by the maps $\iota_Q$ and $\iota_A$ above, and let $N=\kappa^{-1}(N_Q)\cap F_{\hat{Q}}$, where $N_Q$ is the subgroup of $F_Q$ defining $G_\M$ as in Theorem \ref{thm:AlternativeDefBireversibleGroup} \ref{item:MaxNormal}.
From the presentation of $\pi_1(\hat{\M})$ and Theorem \ref{thm:AlternativeDefBireversibleGroup} \ref{item:TwoFreeGroups}, we see that $F_{\hat{Q}}F_A$ is a subgroup of $\pi_1(\hat{\M})$, and since $\kappa$ restricts to the identity on $F_A$, we conclude that $\kappa$ restricts to a surjective homomorphism defined on $F_{\hat{Q}}F_A$ whose kernel is contained in $F_{\hat{Q}}$.
Therefore, $N$ is the pre-image of $N_Q\leq F_Q$ by $\kappa|_{F_{\hat{Q}}F_A}$, which implies that $N$ is normalised by $F_A$.
Similarly, $F_{\hat{Q}}F_{\overline{A}}$ is also a subgroup of $\pi_1(\hat{\M})$, and $N$ is the pre-image of $N_Q$ by $\kappa|_{F_{\hat{Q}}F_{\overline{A}}}$, which implies that it is also normalised by $F_{\overline{A}}$.
It follows that $N$ is normal in $\pi_1(\hat{\M})$ and thus that $N\leq N_{\hat{Q}}$, where $N_{\hat{Q}}$ is the normal subgroup defining $G_{\hat{M}}$ as in Theorem \ref{thm:AlternativeDefBireversibleGroup} \ref{item:MaxNormal}.
To see that $N = N_{\hat{Q}}$, it suffices to notice that $\kappa(N_{\hat{Q}})$ has to be a normal subgroup of $\pi_1(\M)$ that is contained in $\kappa(F_{\hat{Q}}) = F_Q$, so that $\kappa(N_{\hat{Q}})\leq N_Q$.
Therefore, $N_{\hat{Q}} \leq \kappa^{-1}(N_Q)\cap F_{\hat{Q}} = N$.

We have just shown that $N_{\hat{Q}}= \kappa|_{F_{\hat{Q}}}^{-1}(N_Q)$, and since $\kappa|_{F_{\hat{Q}}}$ maps onto $F_Q$, we conclude that it induces an isomorphism $f_{Q}\colon G_{\hat{\M}} = F_{\hat{Q}}/N_{\hat{Q}} \rightarrow G_\M = F_Q/N_Q$.
Since $\kappa(\overline{q}) = \kappa(q)^{-1}$ for all $q\in Q$ by definition, it follows that $p_{\hat{Q}}(\hat{Q}^*) = G_{\hat{\M}}$.

The proofs of the statements regarding $\eG_{\hat{\M}}$ are similar and will be omitted.
\end{proof}

%

\section{Bireversible automata and commensurators of finitely generated groups in the automorphism groups of their Cayley graph}\label{sec:Commensurator}

In this section, we observe that Macedo\'{n}ska, Nekrashevych and Sushchansky's theorem (Theorem \ref{thm:MNS}) can be extended to all finitely generated groups, and we extract a few interesting consequences from this fact.

We begin by remarking that if $(G,A,\varphi)$ is a marked group, the group of elements of $\Aut_1(\ker(\varphi)\backslash \overrightarrow{T}_A)$ commensurating $G$ is isomorphic to a subgroup of elements of $\Aut_1(\overrightarrow{T}_A)$ commensurating $F_A$.

\begin{prop}\label{prop:CommensuratorQuotientGraph}
Let $(G,A,\varphi)$ be a finitely generated marked group and let $X = \ker(\varphi)\backslash \overrightarrow{T}_A$, where $\overrightarrow{T}_A$ is the directed Cayley graph of the free group $F_A$ with respect to $A$.
Let
\[\C{G}= \Comm_{\Aut(X)}(G) \cap \Aut_1(X)\]
be the subgroup of automorphisms of $X$ fixing the vertex corresponding to the identity in $G$ and commensurating $G$, let $N_\varphi$ be the normaliser of $\ker(\varphi)$ in $\Aut(\overrightarrow{T}_A)$ and let
\[\C[N_\varphi]{F_A} = \Comm_{N_\varphi}(F_A)\cap \Aut_1(\overrightarrow{T}_A)\]
be the elements of $N_\varphi$ commensurating $F_A$ and fixing the identity.
Then, when restricted to $\C[N_\varphi]{F_A}$, the homomorphism $\zeta_\varphi$ of Proposition \ref{prop:NormaliserPassesToQuotient} becomes an isomorphism between $\C[N_\varphi]{F_A}$ and $\C{G}$.
\end{prop}
\begin{proof}
By Proposition \ref{prop:NormaliserPassesToQuotient}, it suffices to show that $\zeta_\varphi(\C[N_\varphi]{F_A}) = \C{G}$.
If $g\in N_\varphi$ commensurates $F_A$, then $\zeta_\varphi(g)$ commensurates $\zeta_\varphi(F_A)= G$, and since $\zeta_\varphi(\Aut_1(\overrightarrow{T}_A))\subseteq \Aut_1(X)$, it follows that $\zeta_\varphi(\C[N_\varphi]{F_A}) \subseteq \C{G}$.

Now, let $g\in \C{G}$ be an arbitrary element, and let $\overline{g}\in N_{\Aut(\overrightarrow{T}_A)}(\ker(\varphi))\cap \Aut_1(\overrightarrow{T}_A)$ be the unique element such that $\zeta_\varphi(\overline{g}) = g$, whose existence is guaranteed by Proposition \ref{prop:NormaliserPassesToQuotient}.
Since $\ker(\varphi)$ is normalised by $\overline{g}$, it is contained in $\overline{g}F_A\overline{g}^{-1}$, so that it lies in $F_A\cap \overline{g}F_A\overline{g}^{-1}$.
Since the kernel of $\zeta_\varphi$ is $\ker(\varphi)$, it follows that the index of $F_A\cap \overline{g}F_A\overline{g}^{-1}$ in $F_A$ (respectively $\overline{g}F_A\overline{g}^{-1}$) is equal to the index of
\[\zeta_\varphi(F_A\cap \overline{g}F_A\overline{g}^{-1}) = G \cap g Gg^{-1}\]
in $G$ (respectively $gGg^{-1}$), which is finite since $g$ commensurates $G$ by assumption.
Thus, $\overline{g}\in \C[N_\varphi]{F_A}$, which implies that $\zeta_\varphi(\C[N_\varphi]{F_A}) = \C{G}$.
\end{proof}

From Proposition \ref{prop:CommensuratorQuotientGraph} and Theorem \ref{thm:MNS}, we see that if $(G,A,\varphi)$ is a marked group, the group of automorphisms of $\ker(\varphi)\backslash \overrightarrow{T}_A$ commensurating $G$ and fixing the identity is isomorphic to the subgroup of all bireversible automorphisms of $\overrightarrow{T}_A$ normalising $\ker(\varphi)$.
As we will see, the latter subgroup can be expressed as a directed union of bireversible groups.
To establish this, we first need to introduce a notion of compatibility between a bireversible automaton and a marked group.

\begin{defn}\label{defn:CompatibleAutomaton}
A bireversible automaton $\M=(A,Q,\lambda,\rho)$ is compatible with a marked group $(G,A,\varphi)$ if $\ker(\varphi)$ is a normal subgroup of $\pi_1(\M)/N_Q$, where $N_Q$ is as in Theorem \ref{thm:AlternativeDefBireversibleGroup} \ref{item:MaxNormal}.
\end{defn}

Recall from Section \ref{subsec:CommensuratorFree} that for every bireversible automaton $\M=(A,Q,\lambda,\rho)$, we have an injective homomorphism $\psi_\M \colon \pi_1(\M)/N_Q \rightarrow \Aut(\overrightarrow{T}_A)$.
If $\M$ is compatible with a marked group $(G,A,\varphi)$, then $\psi_\M(\pi_1(\M)/N_Q)$ normalises $\psi_\M(\ker(\varphi)) = \ker(\varphi)$.
In particular, the map $\zeta_\varphi$ of Proposition \ref{prop:NormaliserPassesToQuotient} is defined and injective on $\psi_\M(G_\M)$, so that we can see $G_\M$ as a subgroup of $\Aut_1(\ker(\varphi)\backslash\overrightarrow{T}_A)$.
In analogy with the case of the tree, we will call automorphisms of the graph $\ker(\varphi)\backslash\overrightarrow{T}_A$ in the image of some bireversible group \emph{bireversible automorphisms}.

\begin{defn}
Let $(G,A,\varphi)$ be a finitely generated marked group.
An automorphism $f\in \Aut_1(\ker(\varphi)\backslash \overrightarrow{T}_A)$ is a bireversible automorphism if there exists a bireversible automaton $\M=(A,Q,\lambda,\rho)$ compatible with $(G,A,\varphi)$ such that $f\in \zeta_\varphi(\psi_\M(G_\M))$, where the map $\psi_\M$ is as in Section \ref{subsec:CommensuratorFree} and the map $\zeta_\varphi$ is as in Proposition \ref{prop:NormaliserPassesToQuotient}.
\end{defn}

We are now able to state an analogue of Macedo\'{n}ska, Nekrashevych and Sushchansky's theorem for all finitely generated groups.

\begin{thm}\label{thm:GeneralisedMNS}
Let $(G,A,\varphi)$ be a finitely generated marked group and let $X=\ker(\varphi)\backslash \overrightarrow{T}_A$, where $\overrightarrow{T}_A$ is the directed Cayley graph of the free group $F_A$ with respect to $A$.
Let
\[\C{G}= \Comm_{\Aut(X)}(G) \cap \Aut_1(X)\]
be the subgroup of automorphisms of $X$ fixing the vertex corresponding to the identity in $G$ and commensurating $G$, and let $B_X$ be the set of bireversible automorphisms of $X$.
Then, $\C{G}=B_X$.
\end{thm}
\begin{proof}
By Proposition \ref{prop:CommensuratorQuotientGraph}, $\C{G}=\zeta_\varphi(\C[N_\varphi]{F_A})$, where $\zeta_\varphi$ is the homomorphism of Proposition \ref{prop:NormaliserPassesToQuotient} and
\[\C[N_\varphi]{F_A} = \Comm_{\Aut(\overrightarrow{T}_A)}(F_A)\cap \Aut_1(\overrightarrow{T}_A)\cap N_\varphi\]
with $N_\varphi$ denoting the normaliser in $\Aut(\overrightarrow{T}_A)$ of $\ker(\varphi)$.
Furthermore, by Proposition \ref{prop:NormaliserPassesToQuotient}, the map $\zeta_\varphi$ is an isomorphism onto its image when restricted to elements fixing the identity.
Let $B'_X\subseteq N_{\varphi}\cap\Aut_1(\overrightarrow{T}_A)$ be the pre-image of $B_X$ by the restriction of $\zeta_\varphi$ to $\Aut_1(\overrightarrow{T}_A)$, so that $\zeta_\varphi$ becomes a bijection when restricted to $B'_X$.

By Theorem \ref{thm:MNS}, $\C[N_\varphi]{F_A} = B\cap N_\varphi$, where $B$ denotes the set of all bireversible automorphisms of $\overrightarrow{T}_A$.
Thus, we need only to prove that $B\cap N_\varphi = B'_X$.

It is clear from the definitions that $B'_X\subseteq B\cap N_\varphi$.
To prove the other containment, let $g\in B\cap N_\varphi$ be an arbitrary element.
By definition of $B$, there exists a bireversible automaton $\M=(A,Q,\lambda,\rho)$ and an element $\hat{g}\in G_\M$ such that $g\in \psi_\M(\hat{g})$, where the homomorphism $\psi_\M\colon \pi_1(\M)/N_Q \rightarrow \Aut(\overrightarrow{T}_A)$ is the one described in Section \ref{subsec:CommensuratorFree}.
Then, by Proposition \ref{prop:AutomatonCorrespondingToSubgroup}, there exists a bireversible automaton $\M'=(A,S,\lambda',\rho')$ such that $\psi_{\M'}(\pi_1(\M')/N_S) = \psi_\M(\langle \hat{g}, F_A \rangle) = \langle g, F_A \rangle$.
Since $g\in N_\varphi$ and $F_A\leq N_\varphi$, we conclude that $\psi_{\M'}(\pi_1(\M')/N_S)\leq N_\varphi$, and since $\psi_{\M'}$ is an isomorphism onto its image that sends $F_A$ to $F_A$, we conclude that $\ker(\varphi)$ is normal in $\pi_1(\M')/N_S$.
Therefore, $\M'$ is compatible with the marked group $(G,A,\varphi)$ and thus $g$ is a bireversible automorphism of $X$.
This shows that $B\cap N_\varphi = B'_X$, which implies that $\C{G}=B_X$.
\end{proof}

As a direct consequence of Theorem \ref{thm:GeneralisedMNS}, we recover the well-known fact that the group generated by a bireversible automaton is finite if and only if the dual group is finite (see for instance \cite{SavchukVorobets11} Proposition 2.2, where it is stated in greater generality that what we need here).
We record it here as we shall need it later on.

\begin{cor}\label{cor:GroupFiniteIFFDualFinite}
Let $\M=(A,Q,\lambda,\rho)$ be a bireversible automaton.
Then, $G_\M$ is finite if and only if $\eG_\M$ is finite.
\end{cor}
\begin{proof}
Suppose that $\eG_\M$ is finite, and let $\varphi\colon F_A \rightarrow \eG_\M$ be the canonical quotient map.
Then, the graph $\ker(\varphi)\backslash \overrightarrow{T}_A$ is finite and thus its group of automorphisms is finite.
By Theorem \ref{thm:AlternativeDefBireversibleGroup} \ref{item:MaxNormal}, $\M$ is compatible with $(\eG_\M, A,\varphi)$.
Therefore, $\zeta_\varphi(\psi_\M(G_\M))$ is well-defined, where $\psi_\M$ is the map of Section \ref{subsec:CommensuratorFree} and $\zeta_\varphi$ is the one of Proposition \ref{prop:NormaliserPassesToQuotient}, and it is finite, since it is a subgroup of a finite group.
The map $\psi_\M$ is injective, and $\zeta_\varphi$ is injective when restricted to automorphisms fixing the identity.
We conclude that $G_\M$ is finite.

The fact that $\eG_\M$ is finite if $G_\M$ is follows directly from the above argument and Proposition \ref{prop:DualIsBireversible}.
\end{proof}

Recall that the (unlabelled) Cayley graph of a group $G$ with respect to a finite generating set $S$ is the directed graph whose vertex set is $G$ and whose edge set is $\{(g,gs)\mid g\in G, s\in S\cup S^{-1}\}$.
Thus, as long as one excludes the identity from the generating set, Cayley graphs are a particular case of graphs of the form $\ker(\varphi)\backslash\overrightarrow{T}_A$ and we can apply Theorem \ref{thm:GeneralisedMNS} to them.

%

As a consequence of Theorem \ref{thm:GeneralisedMNS}, we obtain that any group admitting a Cayley graph for which the commensurator of the group in the automorphism group of the Cayley graph is "big" must be closely related to an automaton group, in a sense that we are going to make precise.

\begin{thm}\label{thm:GroupsWithBigCommensurators}
Let $G$ be a finitely generated group, and suppose that there exists some finite generating set $S\subseteq G$ (not containing the identity) such that the group
\[\C[\Cay(G,S)]{G} = \Comm_{\Aut(\Cay(G,S))}(G) \cap \Aut_1(\Cay(G,S))\]
is not locally finite.
Then, there exists a sequence $G\trianglerighteq N_1\trianglerighteq N_2 \trianglerighteq \dots$ of normal subgroups of $G$ such that each $G/N_i$ is an infinite bireversible group and $\bigcap_{i=1}^{\infty} N_i = D_G$, where $D_G$ denotes the largest subgroup of $G$ that is normalised by $\Comm_{\Aut(\Cay(G,S))}(G)$.
In other word, $G/D_G$ is residually infinite bireversible.
\end{thm}
\begin{proof}
Let $A=S\cup S^{-1}$, and let $\varphi\colon F_A\rightarrow G$ be the canonical quotient map.
Then, the Cayley graph $\Cay(G,S)$ is isomorphic to the graph $X=\ker(\varphi)\backslash \overrightarrow{T_A}$, where $\overrightarrow{T}_A$ denotes the oriented Cayley graph of $F_A$ with respect to $A$.
Let us denote by $B_X$ the set of all bireversible automorphisms of $X$.
By Theorem \ref{thm:GeneralisedMNS}, $\C{G} = B_X$, so that $B_X$ is not locally finite.

By definition, there exist (not necessarily distinct) bireversible automata $\M_i$ ($i\in \N_{> 0}$) compatible with the marked group $(G,A,\varphi)$ such that $B_X=\bigcup_{i\in \N}\zeta_\varphi(\psi_{\M_i}(G_{\M_i}))$, where $\psi_{\M_i}$ is as in Section \ref{subsec:CommensuratorFree} and $\zeta_\varphi$ is as in Proposition \ref{prop:NormaliserPassesToQuotient}.
Since $B_X$ is not locally finite, we can assume without loss of generality that $\zeta_\varphi(\psi_{\M_1}(G_{\M_1}))$ is infinite.

Let $\M'_1=\M_1$, and let us recursively define $\M'_i = \M'_{i-1}\sqcup \M_i$ for $i\in \N_{>1}$ (see Proposition \ref{prop:UnionOfAutomata} for the definition of the union of two bireversible automata over the same alphabet).
Since $\ker(\varphi)$ is normal in $\pi_1(\M_i)$ for all $i\in \N_{>0}$, it follows from Proposition \ref{prop:UnionOfAutomata} that it is also normal in $\pi_1(\M'_i)$ for all $i\in \N_{>0}$, so that the automata $\M'_i$ are compatible with $(G,A,\varphi)$.
Still by Proposition \ref{prop:UnionOfAutomata}, $G_{\M'_{i-1}}, G_{\M_i}\leq G_{\M'_i}$ for all $i$, so that $\zeta_\varphi(\psi_{\M'_i}(G_{\M'_i})) \subseteq \zeta_\varphi(\psi_{\M'_{i+1}}(G_{\M'_{i+1}}))$ and $B_X=\bigcup_{i\in \N}\zeta_\varphi(\psi_{\M'_i}(G_{\M'_i}))$.
Note in particular that each $G_{\M'_i}$ is infinite by our assumption on $G_{\M_1}$.

Let us denote by $K_i\trianglelefteq F_A$ the subgroups such that $F_A/K_i = \eG_{\M'_i}$.
By Theorem \ref{thm:AlternativeDefBireversibleGroup} \ref{item:MaxNormal}, $K_i$ is the largest normal subgroup of $\pi_1(\M'_i)$ contained in $F_A$, so that $\ker(\varphi)\leq K_i$.
It follows from Proposition \ref{prop:UnionOfAutomata} and Theorem \ref{thm:AlternativeDefBireversibleGroup} that $\pi_1(\M'_i)\leq \pi_1(\M'_{i+1})$, which implies that $K_{i+1}$ is a normal subgroup of $\pi_1(\M'_i)$ and thus that $K_{i+1}\leq K_i$.
If we write $N_i=K_i/\ker(\varphi)$, we conclude that $G=F_A/\ker(\varphi)\trianglerighteq N_1\trianglerighteq N_2 \trianglerighteq \dots$.
We have $G/N_i \cong F_A/K_i = \eG_{\M'_i}$, so that $G/N_i$ is a bireversible group by Proposition \ref{prop:DualIsBireversible}.
Since $G_{\M'_i}$ is infinite, we conclude from Corollary \ref{cor:GroupFiniteIFFDualFinite} that $\eG_{\M'_i}$, and thus $G/N_i$, is infinite.

Let $K_{\infty} = \bigcap_{i=1}^{\infty}K_i$.
Then, $\ker(\varphi)\leq K_\infty$, and since each $K_i$ is normal in $\pi_1(\M'_{j})$ for $j\leq i$, we conclude that $K_\infty$ is normal in $\pi_1(\M'_i)$ for all $i\in \N_{>0}$.
It follows that $K_\infty$ is normalised by $\psi_{\M'_i}(G_{\M'_i})$, and thus that $\bigcap_{i=1}^{\infty}N_i = K_\infty / \ker(\varphi)$ is normalised by $\zeta_\varphi(\psi_{\M'_i}(G_{\M'_i}))$ for all $i\in \N_{>0}$.
Therefore, $K_\infty/\ker(\varphi)$ is normalised by $B_X=\C{G}$ and is normal in $G$.
Since every element of $\Comm_{\Aut(\Cay(G,S))}(G)$ can be written as a product of an element of $G$ and of an element of $\C{G}$, we conclude that $K_\infty/\ker(\varphi)\leq D_G$, where $D_G$ is the largest subgroup of $G$ normalised by $\Comm_{\Aut(\Cay(G,S))}(G)$.

To see the opposite containment, notice that $D_G$ is normalised by $\zeta_\varphi(\psi_{\M'_i}(G_{\M'_i}))$ for all $i\in \N_{>0}$, which means that $\zeta_{\varphi}^{-1}(D_G)\leq F_A$ is normalised by $\psi_{\M'_i}(G_{\M'_i})$.
Since $D_G$ is also normal in $G$ (because $G\leq \Comm_{\Aut(\Cay(G,S))}(G)$), we conclude that $\zeta_{\varphi}^{-1}(D_G)$ is normal in $\pi_1(\M'_i)$, and thus contained in $K_i$, for every $i\in \N_{>0}$.
This finishes showing that $\bigcap_{i=1}^{\infty}N_i = D_G$.
\end{proof}

As a corollary, we obtain that if every proper quotient of a group is virtually nilpotent (in particular, if every proper quotient is finite, i.e. if the group is just infinite), then the commensurator of this group in the group of automorphisms of any of its Cayley graphs has to be "small", unless the group is a bireversible group.

\begin{cor}\label{cor:JustInfiniteIsBirOrSmall}
If $G$ is an infinite finitely generated group such that every proper quotient is virtually nilpotent, then either $G$ is a bireversible group or 
\[\C{G}= \Comm_{\Aut(X)}(G) \cap \Aut_1(X)\]
is locally finite for all Cayley graphs $X$ of $G$.
\end{cor}
\begin{proof}
If $\C{G}$ is not locally finite, then by Theorem \ref{thm:GroupsWithBigCommensurators}, there exists a normal subgroup $N\trianglelefteq G$ such that $G/N$ is an infinite bireversible group.
By a result of Klimann \cite{Klimann18}, an infinite bireversible group cannot be virtually nilpotent, which by our assumptions implies that $G/N$ cannot be a proper quotient of $G$.
Therefore, $G$ has to be bireversible.
\end{proof}

At present, we do not know if there can exist just infinite bireversible groups.

\begin{qu}
Can a bireversible automaton group be just infinite?
\end{qu}

\section{Bireversible groups and non compact-by-discrete automorphism groups of graphs}\label{sec:GraphicalDiscreteness}

In this section, we observe that an infinite bireversible group $G$ acts faithfully and geometrically on a graph $X$ such that the commensurator of $G$ in the group of automorphisms of $X$ has a finitely generated subgroup that is not discrete.
It follows that $\Aut(X)$ is not compact-by-discrete (recall that a topological group $G$ is \emph{compact-by-discrete} if there exists a compact normal subgroup $N$ such that $G/N$, equipped with the quotient topology, is discrete.).
This will allow us to conclude that several groups are not bireversible groups.


\begin{thm}\label{thm:BirNotGraphicallyDiscrete}
Let $G$ is an infinite bireversible group.
Then, there exists a connected, locally finite graph $X$ on which $G$ acts faithfully and geometrically such that $\text{Comm}_{\Aut(X)}(G)\cap \Aut_x(X)$ contains a finitely generated subgroup that is not discrete, for some $x\in X$. 
Consequently, $\Aut(X)$ is not compact-by-discrete.
\end{thm}
\begin{proof}
By Proposition \ref{prop:DualIsBireversible}, there exists a bireversible automaton $\M=(A,Q,\lambda,\rho)$ such that $G\cong \eG_\M$.
Let us denote by $\varphi\colon F_A \rightarrow \eG_\M$ the canonical quotient map.
Then, by Theorem \ref{thm:AlternativeDefBireversibleGroup} \ref{item:MaxNormal}, $\M$ is compatible with the marked group $(\eG_\M, A, \varphi)$ (see Definition \ref{defn:CompatibleAutomaton}), so that $\zeta_\varphi(\psi_\M(G_\M))\leq \Aut_1(\overrightarrow{X})$, where $\overrightarrow{X}=\ker(\varphi)\backslash \overrightarrow{T}_A$, $\overrightarrow{T}_A$ denotes the oriented Cayley graph of $F_A$ with respect to $A$, $\psi_\M\colon \pi_1(\M)/N_Q \rightarrow \Aut(\overrightarrow{T}_A)$ is the map described in Section \ref{subsec:CommensuratorFree} and $\zeta_\varphi\colon N_{\Aut(\overrightarrow{T}_A)}(\ker(\varphi))\rightarrow \Aut(\overrightarrow{X})$ is the map of Proposition \ref{prop:NormaliserPassesToQuotient}.
By Theorem \ref{thm:GeneralisedMNS}, $\zeta_\varphi(\psi_\M(G_\M))\leq \C[\overrightarrow{X}]{G} = \text{Comm}_{\Aut(\overrightarrow{X})}(G) \cap \Aut_1(\overrightarrow{X})$.
As noted in Section \ref{subsec:CommensuratorFree}, $\psi_\M$ is injective, and $\zeta_\varphi$ is injective on the image of $\psi_\M(G_\M)$ by Proposition \ref{prop:NormaliserPassesToQuotient}.
Therefore, $\C[\overrightarrow{X}]{G}$ contains a subgroup isomorphic to $G_\M$.
Note that since $\eG_\M$ is infinite, $G_\M$ is infinite by Corollary \ref{cor:GroupFiniteIFFDualFinite}, so that $\C[\overrightarrow{X}]{G}$ contains a finitely generated infinite subgroup.

Let us denote by $X$ the unoriented graph obtained from $\overrightarrow{X}$ by forgetting the orientation of the edges.
It is clear that $\Aut(\overrightarrow{X})\leq \Aut(X)$, so that $\C{G} = \text{Comm}_{\Aut(X)}(G)\cap \Aut_1(X)$ contains a finitely generated infinite subgroup.
Since $\Aut(X)$ is Hausdorff, any discrete subgroup has to be closed, and since $\Aut_1(X)$ is compact (see for instance \cite{Woess91}, Lemma 1), all discrete subgroups of $\Aut_1(X)$ have to be finite.
We conclude that $\text{Comm}_{\Aut(X)}(G)\cap \Aut_1(X)$ contains a finitely generated subgroup that is not discrete.

To see that $\Aut(X)$ is not compact-by-discrete, let us suppose for the sake of contradiction that it is.
Then, there exists a compact normal subgroup $N\trianglelefteq \Aut(X)$ such that $\Aut(X)/N$ is discrete with the quotient topology.
By compactness, the orbit of any vertex under the action of $N$ must be finite (this is due to the fact that cosets of vertex stablilisers are open sets in $\Aut(X)$).
Since the action of $\Aut(X)$ is transitive on the vertices of $X$ (the action of $G\cong \eG_\M$ on $X$ by left multiplication is transitive on the set of vertices, which is $\eG_\M$), and since $N$ is normal in $\Aut(X)$, all orbits of vertices under the action of $N$ are conjugate, and thus must have the same size.
This implies that $N$ is locally finite, since it can be mapped into a direct power of finite symmetric groups of bounded size (one for each orbit), and the kernel corresponds to the subgroup of $N$ fixing all the vertices of $X$, which must also be locally finite, since it embeds in a product of finite symmetric groups of bounded size (one for each pair of vertices with multiple edges between them).

Let us denote by $\pi\colon \Aut(X) \rightarrow \Aut(X)/N$ the quotient map, which is continuous by definition.
Since $\Aut_1(X)$ is compact, the image $\pi(\Aut_1(X))$ is compact in the discrete group $\Aut(X)/N$, and thus finite.
This implies that $\Aut_1(X)\cap N$ is of finite index in $\Aut_1(X)$.
Therefore, since $N$ is locally finite, $\Aut_1(X)$ must also be locally finite, and thus every finitely generated subgroup of $\Aut_1(X)$ must be discrete.
This contradicts the fact that $\C{G}$ contains a finitely generated non-discrete subgroup.
Thus, $\Aut(X)$ is not compact-by-discrete.

Since $G$ acts freely and transitively on the vertex set of $X$ by left multiplication, it acts geometrically on $X$ when the latter is seen as a 1-dimensional CW-complex.
\end{proof}

%
%

Theorem \ref{thm:BirNotGraphicallyDiscrete} gives us a convenient criterion to establish that a group cannot be bireversible.
We give below a non-exhaustive list of groups which can be shown not to be bireversible by combining Theorem \ref{thm:BirNotGraphicallyDiscrete} with other results in the literature.

\begin{cor}\label{cor:ListOfGroups}
Let $G$ be a finitely generated group.
If $G$ belongs to one of the following families of groups
\begin{enumerate}
\item infinite virtually nilpotent groups,\label{item:virtuallyNilpotent}
\item irreducible lattices in center-free, real semi-simple Lie groups without compact factors and not locally isomorphic to $\mathrm{SL}_2(\mathbb{R})$,
\item uniform lattices in $\mathrm{PSL}_2(\mathbb{R})$,
\item the group $\mathrm{Out}(F_n)$ of outer automorphisms of a free group of rank $n\geq 3$,
\item topologically rigid hyperbolic groups,
\item hyperbolic groups whose visual boundary is homeomorphic to an $n$-sphere with $n\leq 3$ or to a Sierpinski carpet,
\item fundamental groups of closed irreducible oriented 3-manifolds with non-trivial geometric decomposition,
\end{enumerate}
then $G$ is not bireversible.
\end{cor}
\begin{proof}
We refer the reader to \cite{MargolisSherpherStarkWoodhouse23+} and the references cited therein for proofs that the family of groups in the statement cannot act geometrically on connected locally finite graphs with a non compact-by-discrete group of automorphisms.
Thus, they are not bireversible automaton groups by Theorem \ref{thm:BirNotGraphicallyDiscrete}.
\end{proof}

The fact that virtually nilpotent groups are not bireversible automaton groups (item \ref{item:virtuallyNilpotent}) was already shown by Klimann in \cite{Klimann18} by a different method, but to the best of our knowledge, there were no prior proofs for the other items on this list.

\section{Cyclic subgroup distortion}\label{sec:Distortion}

Our main goal in this section is to show that cyclic subgroups of bireversible automaton groups are always undistorted, from which we will derive a few consequences.
Before we do that, let us first recall what it means for a subgroup to be undistorted.

\begin{defn}
Let $G$ be a finitely generated group, generated by some finite set $S\subseteq G$, and let $H\leq G$ be a finitely generated subgroup of $G$, generated by some finite set $T\subseteq H$.
We say that $H$ is \emph{undistorted} in $G$ if there exists some constant $C\in \mathbb{R}_{>0}$ such that
\[|h|_T \leq C|h|_S\]
for all $h\in H$, where $|\cdot|_T$ denotes the word metric in $H$ with respect to $T$ and  $|\cdot|_S$ denotes the word metric in $G$ with respect to $S$.
\end{defn}

In other words, the subgroup $H$ is undistorted in $G$ if the intrinsic metric in $H$ is equivalent to the metric inherited from $G$.
It is easy to check that being undistorted is independent of the choice of generating sets $S$ and $T$.

In what follows, we will be interested in distortion for cyclic subgroups.
We establish below an elementary characterisation that will slightly simplify he argument in that case.
We include a proof for completeness.

\begin{lemma}\label{lemma:EasierDistortionCheck}
Let $G$ be a finitely generated group, $S$ be a finite symmetric generating set and $g\in G$ be an element of infinite order.
Then, $H=\langle g \rangle$ is undistorted in $G$ if and only if there exist constants $C\in \mathbb{R}_{>0}$ and $k\in \N_{>0}$ such that
\[kn\leq C|g^{kn}|_S\]
for all $n\in \N$.
\end{lemma}
\begin{proof}
It is clear that if $H$ is undistorted, then the inequality holds.
For the converse, let us assume that it holds, let $T=\{g\}$ be our generating set for $H$, and let $h\in H$ be any element.
Then, there exists $m\in \mathbb{Z}$ such that $h=g^{m}$.
First of all, notice that $|g^{m}|_S = |g^{-m}|_S$, so that we can assume without loss of generality that $m\in \N$.
Let us write $m=kn+r$ for some $k,r\in \N$ with $0\leq r<n$.
Then,
\[|g^{m}|_T = m = kn+r \leq C|g^{kn}|_S+r.\]
We have $|g^{kn}|_S \leq |g^{kn+r}|_S + |g^{-r}|_S = |g^{m}|_S + |g^r|_S$,
so that
\[|g^{m}|_T \leq C|g^{m}|_S + C|g^r|_S + r.\]
Since $C|g^r|_S + r$ can be bounded from above by some constant, and since $|g^{m}|_S\geq 1$ unless it is the identity, we conclude that there exists $C'\in \mathbb{R}_{>0}$ such that $|g^m|_T \leq C'|g^m|_S$ for all $m\in \mathbb{Z}$.
\end{proof}

To be able to study distortion in bireversible groups, we will need to make use of the following theorem, whose proof can be extracted from the proof of the main result of \cite{FrancoeurMitrofanov21}.

\begin{thm}[c.f. \cite{FrancoeurMitrofanov21}, Theorem 1.1]\label{thm:TechnicalFrancoeurMitrofanovThm}
Let $\M=(A,Q,\lambda, \rho)$ be a bireversible automaton such that $p_Q(Q^*)= G_\M$ and $p_A(A^*) = \eG_\M$, where $p_Q\colon F_Q\rightarrow G_\M$ and $p_A\colon F_A \rightarrow \eG_\M$ are the canonical quotient maps.
Let $v\in A^*$ be such that $p_A(v)$ is an element of infinite order, and let
\[L = \left\{\Psi_\M(\gamma)(v^n) \in A^* \mid n\in \N, \gamma\in Q^* \right\}\]
be the union of the orbits of $v^n$ under the action of $\Psi_\M(Q^*)$, where $\Psi_\M\colon \pi_1(\M)\rightarrow \Aut(\overrightarrow{T}_A)$ is the homomorphism described in Section \ref{subsec:CommensuratorFree} and $\overrightarrow{T}_A$ is the oriented Cayley graph of $F_A$.
Then, there exist $x_1, x_2 \in L$ such that
\begin{enumerate}
\item the submonoid $\langle x_1, x_2 \rangle_+ \leq A^*$ is free on $\{x_1, x_2\}$ and contained in $L$,
\item the map $p_A \colon F_A \rightarrow \eG_\M$ is injective on $\langle x_1, x_2 \rangle_+$.
\end{enumerate}
\end{thm}
\begin{proof}
The proof can be extracted from the proof of Proposition 4.1 of \cite{FrancoeurMitrofanov21}.
Indeed, in Section 4.2 of \cite{FrancoeurMitrofanov21}, a finite-state automaton (called there $\Gamma$) is constructed that can recognise the language $L$, with all states accepting.
By Lemma 3.1 of \cite{FrancoeurMitrofanov21}, if any other vertex of the graph $\Gamma$ is used as a starting state, the language accepted by this new finite-state automaton is a sub-language of $L$.
Then, in Section 4.2 of \cite{FrancoeurMitrofanov21}, it is established that $\Gamma$ possesses a strongly connected component $R$ such that all outgoing edges from any vertex in $R$ lead back into $R$.
Combined together, these facts yield several submonoids of $A^*$ contained in $L$, namely those generated by cycles based at elements of $R$.
From there, the following lemmas in \cite{FrancoeurMitrofanov21} derive a contradiction to the assumption that there exist no free two-generated submonoid of $\eG_\M$.
However, this assumption is only used in Lemma 4.3 of \cite{FrancoeurMitrofanov21}, and what is truly used in the proof of that lemma is the assumption that if $x_1, x_2 \in L$ are two elements such that $\langle x_1, x_2 \rangle_+ \leq L$, then $\langle p_A(x_1), p_A(x_2) \rangle_+$ is not a free submonoid on $\{p_A(x_1), p_A(x_2)\}$.
Thus, using the exact same proof as in \cite{FrancoeurMitrofanov21}, we derive a contradiction to that assumption, which implies that there exist $x_1, x_2 \in L$ such that $\langle x_1, x_2 \rangle_+ \leq L$ and $\langle p_A(x_1), p_A(x_2) \rangle_+$ is free on $\{p_A(x_1), p_A(x_2)\}$.
It follows that $\langle x_1, x_2 \rangle_+$ is also free on $\{x_1,x_2\}$ and that the map $p_A$ is injective on this submonoid.
\end{proof}

Using this, we can study distortion for cyclic subgroups of bireversible groups.

\begin{thm}\label{thm:UndistortedCyclicSubgroups}
Every cyclic subgroup of a bireversible group is undistorted.
\end{thm}
\begin{proof}
Let $G$ be a bireversible group and let $\M=(A,Q,\lambda, \rho)$ be a bireversible automaton such that $p_Q(Q^*)= G_\M$ and $p_A(A^*) = \eG_\M \cong G$, where $p_Q\colon F_Q\rightarrow G_\M$ and $p_A\colon F_A \rightarrow \eG_\M$ are the canonical quotient maps. 
Note that such an automaton exists by Propositions \ref{prop:DualIsBireversible} and \ref{prop:MonoidGenerates}.
We need to show that every cyclic subgroup of $\eG_\M$ is undistorted.
We will fix $S=p_A(A)$ as our finite generating set for $\eG_\M$.

Let $g\in \eG_\M$ be any element.
If $g$ is of finite order, then the subgroup $\langle g \rangle\leq \eG_\M$ is finite and thus obviously undistorted, so we may assume that $g$ is of infinite order.
In that case, let $v \in A^*$ be such that $p_A(v)=g$, and let
\[L = \left\{\Psi_\M(\gamma)(v^n) \in A^* \mid n\in \N, \gamma\in Q^* \right\},\]
where $\Psi_\M\colon \pi_1(\M)\rightarrow \Aut(\overrightarrow{T}_A)$ is the homomorphism described in Section \ref{subsec:CommensuratorFree} and $\overrightarrow{T}_A$ is the oriented Cayley graph of $F_A$.
By Theorem \ref{thm:TechnicalFrancoeurMitrofanovThm}, there exist $x_1, x_2 \in L$ such that $\langle x_1, x_2\rangle_+$ is a free submonoid of $A^*$ contained in $L$, and such that the map $p_A$ is injective when restricted to it.
Without loss of generality, we may assume that $x_1$ and $x_2$ are words of the same length.
Indeed, if that is not the case, then one can simply consider instead $x_1x_2$ and $x_2x_1$, which are words of the same length that will still generate a free submonoid.

Since $\langle x_1, x_2 \rangle_+\subseteq L$, there must exist $\gamma_1, \gamma_2 \in Q^*$ and $k_1, k_2\in \N$ such that $x_i = \Psi_\M(\gamma_i)(v^{k_i})$.
Notice that since $\Psi_\M(\gamma_i)$ are automorphisms of $\overrightarrow{T}_A$ fixing the identity and since $x_1$ and $x_2$ are at the same distance from the identity by assumption, we must in fact have $k_1=k_2$.
Similarly, if $x=x_{i_1}x_{i_2}\dots x_{i_n}$ is an element of $\langle x_1, x_2 \rangle_+\subseteq L$ that can be written as a product of $n$ generators (for some $n\in \N$), then there must exist some $\gamma_x \in Q^*$ and $k_x$ such that $x=\Psi_\M(\gamma_x)(v^{k_x})$.
Note that since $x_1$ and $x_2$ lie in the free semigroup $A^*$, the length of $x$ must be exactly $k_1|v|n$, and since the action of $\Psi_\M(\gamma_x)$ is by isometries fixing the identity, we must have $k_x = k_1n$.
By Theorem \ref{thm:AlternativeDefBireversibleGroup} \ref{item:MaxNormal}, $\ker(p_A)$ is normal in $\pi_1(\M)$, so that $\Psi_\M(F_Q)$ lies in $N_{\Aut(\overrightarrow{T}_A)}(\ker(p_A))$, the normaliser of $\ker(p_A)$ in $\Aut(\overrightarrow{T}_A)$.
Therefore, if we denote by $\zeta_{p_A}\colon N_{\Aut(\overrightarrow{T}_A)}(\ker(p_A))\rightarrow \Aut(\ker(p_A)\backslash \overrightarrow{T}_A)$ the homomorphism of Proposition \ref{prop:NormaliserPassesToQuotient}, we have
\[p_A(x)=p_A(\Psi_\M(\gamma_x)(v^{k_1n})) = \zeta_{p_A}\circ\Psi_\M(\gamma_x)(p_A(v^{k_1n})) = \zeta_{p_A}\circ\Psi_\M(\gamma_x)(g^{k_1n}).\]
Thus, $p_A(x)$ is the image of $g^{k_1n}$ by an automorphism of $\ker(p_A)\backslash \overrightarrow{T}_A$ fixing the identity, and thus these two elements are at the same distance from the identity in this graph.
Note that the distance between a vertex $\kappa\in \eG_\M$ and the identity in the graph $\ker(p_A)\backslash\overrightarrow{T}_A$ is exactly $|\kappa|_S$, where $|\cdot|_S$ denotes the word metric in $\eG_\M$ with respect to the generating set $S$.
Therefore, we have just shown that if $x$ is an element of length $n$ in the semigroup $\langle x_1, x_2 \rangle_+$, then $|p_A(x)|_S = |g^{k_1n}|_S$.



For $n\in \N$, let us denote by $B_S(n) = \{h\in \eG_\M \mid |h|_S\leq n\}$ the ball of radius $n$ in $\eG_\M$, and by $F(n) = \{x_{i_1}\dots x_{i_n} \mid i_j\in \{1,2\} \} \subseteq \langle x_1, x_2\rangle_+ \leq A^*$ the subset of all elements of $\langle x_1, x_2 \rangle_+$ that can be written as a product of exactly $n$ generators.
Since $\langle x_1, x_2 \rangle_+$ is a free monoid on $\{x_1, x_2\}$, we know that $|F(n)|=2^n$.
Since $p_A$ is injective on $\langle x_1, x_2 \rangle_+$, and since $p_A(x) \in B_S(|g^{k_1n}|_S)$ for all $x\in F(n)$ by the above remark, we have $2^n \leq |B_S(|g^{k_1n}|_S)|$.
On the other hand, since $\eG_\M$ is generated by $S$, we must have $|B_S(|g^{k_1n}|_S)|\leq (2|S|)^{|g^{k_1n}|_S}$, from which we deduce that $2^n \leq (2|S|)^{|g^{k_1n}|_S}$.
The logarithm being an increasing function, we find
\[n \leq \frac{\ln(2|S|)}{\ln(2)}|g^{k_1n}|_S,\]
so that
\[k_1n \leq \frac{\ln(2|S|)}{k_1\ln(2)}|g^{k_1n}|_S.\]
We conclude that $\langle g \rangle$ is undistorted in $\eG_\M$ by Lemma \ref{lemma:EasierDistortionCheck}.
\end{proof}

The previous theorem tells us that any group that contains a distorted cyclic subgroup cannot be generated by a bireversible automaton.
By combining this with other results, we can use this to show that the set of groups generated by bireversible automaton is strictly smaller than the set of groups generated by invertible and reversible automaton.

\begin{cor}\label{cor:InvRevNotBir}
Let $\mathbf{Bir}$ be the set of groups generated by bireversible automata and let $\mathbf{InvRev}$ be the set of groups generated by invertible and reversible automata.
Then, $\mathbf{Bir}\subsetneq \mathbf{InvRev}$.
\end{cor}
\begin{proof}
As every bireversible automaton is invertible and reversible by definition, we have $\mathbf{Bir}\subseteq \mathbf{InvRev}$.
However, by \cite{BartholdiSunic06} Proposition 5, the Baumslag-Solitar groups $BS(1,m)$ belong to $\mathbf{InvRev}$ for all $m\geq 2$.
It is well-known that these groups have distorted cyclic subgroups.
Therefore, they cannot belong to $\mathbf{Bir}$ by Theorem \ref{thm:UndistortedCyclicSubgroups}.
\end{proof}

Combining Theorem \ref{thm:UndistortedCyclicSubgroups} with Theorem \ref{thm:GeneralisedMNS}, we can conclude that for any finitely generated group $G$ and any Cayley graph $X$ of this group, cyclic subgroups in $\Comm_{\Aut(X)}(G)\cap \Aut_1(X)$ are undistorted.
However, for this to make sense, we need to define a version of distortion that makes sense for groups that are not finitely generated.

Notice that if $H\leq K \leq G$ are all finitely generated groups, then $H$ being undistorted in $G$ implies that $H$ is undistorted in $K$.
Thus, if a group $G$ is not finitely generated, it is natural to say that a finitely generated subgroup $H\leq G$ is undistorted in $G$ if for all finitely generated subgroup $K$ of $G$ containing $H$, $H$ is undistorted in $K$.


\begin{cor}\label{cor:UndistortedCyclicInCommensurator}
Let $G$ be a finitely generated group and let $X$ be a Cayley graph for $G$.
Let us denote by
\[\C{G} = \Comm_{\Aut(X)}(G)\cap \Aut_1(X)\]
the subgroup of automorphisms of $X$ fixing the identity and commensurating $G$.
Then, every cyclic subgroup of $\C{G}$ is undistorted.
\end{cor}
\begin{proof}
Let $S\subseteq G$ be a finite generating set and let $X=\Cay(G,S)$ be the associated Cayley graph.
Let $A=S\cup S^{-1}$, and let $\varphi\colon F_A \rightarrow G$ be the canonical quotient map.
Then, $X = \ker(\varphi)\backslash\overrightarrow{T}_A$, where $\overrightarrow{T}_A$ denotes the oriented Cayley graph of $F_A$.
Therefore, by Theorem \ref{thm:GeneralisedMNS}, $\C{G}=B_X$, where $B_X$ denotes the subgroup of all bireversible automorphisms of $X$.

Let $g\in B_X$ be any element, and let $H\leq B_X$ be a finitely generated subgroup containing $g$.
We need to show that $\langle g \rangle$ is undistorted in $H$.
It follows from the proof of Theorem \ref{thm:GroupsWithBigCommensurators} that there exists a bireversible automaton $\M=(A,Q,\lambda, \rho)$ compatible with $(G,A,\varphi)$ such that $H$ is isomorphic to a subgroup of $\eG_\M$.
By Theorem \ref{thm:UndistortedCyclicSubgroups}, $\langle g \rangle$ is undistorted in $\eG_\M$, which means that it is undistorted in $H$.
\end{proof}

\bibliographystyle{plain}
\bibliography{biblio}

\end{document}